\documentclass[11pt]{amsart}

\usepackage{epsfig,color}
\usepackage{xy}
\usepackage{array}
\xyoption{all}

\newtheorem{lemma}{Lemma}[section]
\newtheorem{theorem}[lemma]{Theorem}
\newtheorem{corollary}[lemma]{Corollary}

\newtheorem{proposition}[lemma]{Proposition}
\newtheorem{example}[lemma]{Example}
\newtheorem{remark}[lemma]{Remark}
\newtheorem*{thm}{Theorem}

\renewcommand{\H}{\mathcal H}

\newcommand{\Hom}{\operatorname{Hom}}

\newcommand{\Ext}{\operatorname{Ext}}

\newcommand{\End}{\operatorname{End}}
\newcommand{\op}{{\textit{op}}}

\newcommand{\add}{\operatorname{add}}
\newcommand{\module}{\operatorname{mod}}

\newcommand{\M}{\mathcal M}
\newcommand{\G}{\mathcal G}
\newcommand{\R}{\mathcal R}
\newcommand{\T}{\mathcal T}
\newcommand{\U}{\mathcal U}
\newcommand{\V}{\mathcal V}
\newcommand{\W}{\mathcal W}
\newcommand{\C}{\mathcal C}

\newcommand{\X}{\mathcal X}
\renewcommand{\S}{\mathcal S}

\title{From triangulated categories to module categories via localisation}

\author[Buan]{Aslak Bakke Buan}
\address{Institutt for matematiske fag\\
Norges teknisk-naturvitenskapelige universitet\\
N-7491 Trondheim\\
Norway}
\email{aslakb@math.ntnu.no}

\author[Marsh]{Robert J. Marsh}
\address{School of Mathematics \\ University of Leeds \\ Leeds \\ LS2
  9JT \\ UK}
\email{marsh@maths.leeds.ac.uk}

\thanks{This work was supported by the Engineering and Physical
Sciences Research Council [grant number EP/G007497/1] and by the NFR
[FRINAT grant number 196600].}

\begin{document}

\begin{abstract}
We show that the category of finite-dimensional modules over the
endomorphism algebra of a rigid object in a Hom-finite triangulated category
is equivalent to the Gabriel-Zisman localisation of the category with respect
to a certain class of maps.
This generalises the $2$-Calabi-Yau tilting theorem of Keller-Reiten,
in which the module category is obtained as a factor category,
to the rigid case.
\end{abstract}


\date{22 February 2011}


\maketitle

\section*{Introduction}
Localisation
is an important tool in category theory. At a fundamental level,
it is the introduction of formal inverses for a class of morphisms in the
category known as 
Gabriel-Zisman localisation~\cite[Chap.\ 1]{gz}. This
category always exists provided there are no set-theoretic obstructions.
Morphisms in the new category can be regarded as compositions of the original morphisms and the formal inverses that were added (up to a certain equivalence relation). If the class of morphisms to be inverted satisfies certain axioms
then the new morphisms can be described via a calculus of fractions:
in particular, each new morphism can be written as the composition of one of
the original morphisms and a formal inverse, often represented as a diagram of
morphisms known as a roof; see~\cite[Sect.\ 3]{krause}. This is, for example, the case when the bounded derived category of a module category is formed from the homotopy category of complexes (see e.g.~\cite[III.2-4]{gm}).
In this case the localisation of a
triangulated category gives us a new
triangulated category, which is often the case in applications of
localisation to triangulated categories; see~\cite{krause}.

Here, we present an interesting case where localisation produces instead
an abelian category, in fact the module category over the (opposite) endomorphism algebra of a rigid object in a triangulated category satisfying some mild assumptions. Furthermore, the class of morphisms inverted does not satisfy the axioms
mentioned above.

We consider the tilting theory of cluster categories and, more generally, Hom-finite Calabi-Yau triangulated categories, which has recently been widely investigated. The study of such categories was originally motivated by
their links to cluster algebras, and indeed there has been a
considerable amount of activity and a lot of results in this direction;
see~\cite{k,r} for recent surveys.

However, the study of such categories has also contributed to new
developments in the theory of finite dimensional (and more generally, non-commutative) algebras. One of the central theorems in this direction is the following:
\vskip 0.2cm
\noindent \textbf{$\mathbf{2}$-Calabi-Yau tilting theorem} (Keller-Reiten~\cite[Prop. 2.1]{kr}). \\
\emph{Let $\C$ be a triangulated Hom-finite Krull-Schmidt $2$-Calabi-Yau
category over an algebraically-closed field $k$,
and let $T$ be a cluster-tilting object in $\C$.
Then the category $\C / \Sigma T$ is equivalent to the category
$\module \End_{\C}(T)^{\op}$ of finite dimensional $\End_{\C}(T)^{\op}$-modules.}
\vskip 0.2cm
Here $\Sigma$ denotes the suspension functor of $\C$.
Recall that a \emph{cluster-tilting object} in $\C$ is an object
satisfying $\Ext_{\C}^1(T,X) = 0$ for an object $X$ of $\C$ if and
only if $X$ is in the additive closure of $T$. Here, for objects $C_1$
and $C_2$ in $\C$, $\Ext^1_{\C}(C_1,C_2)$ is shorthand for
$\Hom_{\C}(C_1,\Sigma C_2)$, following the usual convention.

The category $\C /  \Sigma T$ has the same objects as $\C$,
with maps given by maps in $\C$ modulo maps
factoring through $\add \Sigma T$ (the additive closure of $\Sigma T$).
This generalises a corresponding result of~\cite[Thm. 2.2]{bmr} in the case
of cluster categories.

It turns out that the assumption that $\C$ is $2$-Calabi-Yau
is not required for this theorem. This is a result of
Koenig-Zhu~\cite[Cor. 4.4]{kz}. See also~\cite[Prop. 6.2]{iy}.

Our aim in this paper is to use localisation to generalise this result
to the case where $T$ is only assumed to be a rigid object, i.e.\ an object
satisfying $\Ext^1_{\C}(T,T)=0$.
We consider a triangulated Hom-finite Krull-Schmidt triangulated category $\C$,
over a field $k$.
In addition, we assume that $\C$ is skeletally small (as a way of ensuring that
there are no set-theoretic obstructions to the localisations we use).
Note that this condition is satisfied for cluster categories.
We assume that $T$ is a rigid object in $\C$.
We show that with these assumptions, $\module \End_{\C}(T)^{\op}$ can be obtained as the Gabriel-Zisman localisation of $\C$ with respect to a suitable class of
maps in $\C$.

More precisely, let $\X_T$ denote the full subcategory of $\C$ whose
objects are the objects $X$ of $\C$ having no non-zero maps from $T$.
Consider the class of maps $\widetilde{\S}$ in $\C$ consisting of the maps
$f \colon X \to Y$ such that when $f$ is completed to a triangle
$$\Sigma^{-1}Z \overset{h}{\to} X \overset{f}\to Y \overset{g}{\to} Z$$
both $g$ and $h$ factor through $\X_T$. Note that $\widetilde{\S}$ is
well-defined (see Lemma~\ref{Stildewelldefined}).

Let $\C_{\widetilde{\S}}$ be the Gabriel-Zisman localisation \cite[Chapter 1]{gz}
of $\C$ with respect to $\widetilde{\S}$. This is defined by formally inverting
all maps in  $\widetilde{\S}$.
Let $L_{\widetilde{\S}} \colon \C \to \C_{\widetilde{\S}}$ denote the
localisation functor.
We prove that the
functor $H= \Hom_{\C}(T,-) \colon \C \to \module \End_{C}(T)^{\op}$
inverts the maps in $\widetilde{\S}$. Therefore,
there is a uniquely defined functor $G \colon \C_{\widetilde{\S}} \to
\module \End_{C}(T)^{\op}$, such that $H = G L_{\widetilde{\S}}$
Our main theorem is:

\begin{thm} \label{mainresult}
Let $\C$ be a skeletally small Hom-finite Krull-Schmidt triangulated category with
rigid object $T$.
Let $\widetilde{\S}$ be the class of maps defined above. Then the induced functor
$G \colon \C_{\widetilde{\S}} \to \module \End_{C}(T)^{\op}$ is an equivalence.
\end{thm}

We remark that there is another recent approach to the construction of abelian
categories from triangulated categories (as subquotients) by Nakaoka
\cite{nakaoka} (we discuss this in Section $6$).

The paper is organised as follows. In Section $1$ we recall some results
that we need to use, including a triangulated version of Wakamatsu's Lemma.
In Section $2$, we study the Hom-functor associated with a rigid object $T$.
A key result in this section is that a map in $\C$ is inverted by $\Hom_{\C}(T,-)$
if and only if it lies in $\widetilde{\S}$.
In Section $3$, we study the main properties of the Gabriel-Zisman
localisation of $\C$ at $\widetilde{\S}$. In Section $4$ we prove our main result, and in
Section $5$ we describe its relationship to the result of Iyama-Yoshino mentioned
above. In Section $6$ we give a short discussion of the relationship
to the work of Nakaoka \cite{nakaoka}.
In Section $7$ we give some examples.

We would like to thank Yann Palu for his simpler proofs of Lemmas~\ref{factoring}
and~\ref{smallerkernel} and his comments on an earlier version of the
paper.
We would also like to thank Bill Crawley-Boevey and Henning Krause for 
useful discussions and the referee for helpful comments.
Aslak Buan would like to thank the Algebra, Geometry and
Integrable Systems Group and the School of
Mathematics at the University of Leeds for their kind hospitality.
Robert Marsh would like to thank the Department of Mathematical Sciences at the NTNU,
Trondheim, for their kind hospitality.

\section{Preliminaries}

Let $\C$ be a triangulated Hom-finite Krull-Schmidt
category over field $k$.
Let $\Sigma$ denote the suspension functor of $\C$.
%We then have functorial isomorphisms
%$$\Hom_{\C}(X,Y) \simeq D\Hom_{\C}(Y, \Sigma^2 X),$$ where
%$D= \Hom_k(\ , k)$ is the ordinary $k$-duality.

Let $\X$ be a full subcategory of $\C$.
Then a \emph{right $\X$-approximation} of $C$ in $\C$ is a map
$X \to C$, with $X$ in $\X$, such that for all objects $Y$ in $\X$,
the sequence $\Hom_{\C}(Y, X) \to \Hom_{\C}(Y,C) \to 0$ is exact.
A map $X \overset{f}{\to} C$ is
called \emph{right minimal} if, for every $X \overset{g}{\to}  X$ such that $fg=f$,
we have that $g$ is an isomorphism. A map is called a \emph{minimal right $\X$-approximation}
of $C$ if it is right minimal and it is a right $\X$-approximation of $\C$.
Dually, we have the concepts of \emph{left $\X$-approximations} and
\emph{left minimal} maps.

The full subcategory $\X$ is called \emph{functorially finite} if,
for every object $C$ in $\C$,  there exists a right $\X$-approximation
ending in $C$ and a left $\X$-approximation starting in $\C$.

The following is well-known and straightforward to check.

\begin{lemma}\label{app-lemma}
Let $\X$ be a full additive subcategory of $\C$, and $C$ an object of $\C$.
\begin{itemize}
\item[(a)] If there is a right $\X$-approximation of $C$
  then there is a minimal right $\X$-approximation of $C$,
  unique up to isomorphism.
\item[(b)]
If $g  \colon X \to C$ is a minimal right $\X$-approximation of $C$,
then each right approximation is, up to isomorphism, of the form
$f \amalg 0 \colon X \amalg X' \to C$.
\end{itemize}
\end{lemma}

A full subcategory $\X$ of $\C$ is called \emph{extension-closed} if, for
each triangle $X' \to Y \to X'' \to \Sigma X' $ in $\C$ with $X',X''$
in $\X$, also $Y$ is in $\X$. Let $\X^{\perp} = \{Y \in \C \mid
\Ext^{1}_{\C}(X,Y) =0 \text{ for all } X \in \X \}$, and dually let
 $^{\perp} \X = \{Y \in \C \mid
\Ext^{1}_{\C}(Y,X) =0 \text{ for all } X \in \X \}.$

The next lemma is well-known and is a triangulated version of
\emph{Wakamatsu's Lemma}; see e.g.~\cite[Section 2]{iy}; see
also~\cite[Lemma 2.1]{j}.

\begin{lemma}\label{wakamatsu}
Let $\X$ be an extension-closed subcategory of a triangulated category
$\C$.
\begin{itemize}
\item[(a)] Let $X \to C$ be a right $\X$-approximation of $C$ and
$\Sigma^{-1} C \to Y \to X \to C $ a completion to a triangle.
Then $Y$ is in $\X^{\perp}$, and the map
 $\Sigma^{-1} C \to Y$ is a left $\X^{\perp}$-approximation of $\Sigma^{-1} C$.

\item[(b)]  Let $C \to X$ be a left $\X$-approximation of $C$ and
  $\Sigma^{-1}Z \to C \to X \to Z \to \Sigma C $ a completion to a triangle.
Then $Z$ is in $^{\perp}\X$, and the map
 $Z \to \Sigma C$ is a right $^{\perp}\X$-approximation of $\Sigma C$.
\end{itemize}
\end{lemma}

We shall also often use the fact that in the situation of Lemma~\ref{wakamatsu}(b),
the map $\Sigma^{-1}Z\to C$ is a right $\Sigma^{-1}({}^\perp\X)$-approximation of $C$.
%\begin{proof}
%See \cite{iy}[Prop. 2.3].
%\end{proof}

\section{Hom-functors associated to rigid objects}\label{section-rigid}

Let $T$ be a rigid object in $\C$. Let $T^{\perp}$ denote
$(\add T)^{\perp}$, where $\add T$ denotes the additive closure of $T$;
we define ${}^\perp T$ similarly. 
Note that $\Sigma T^{\perp} = \X_T$, as defined in the introduction.
It is easy to see that both $\add T$ and $T^{\perp}$ are extension closed.
In this section we will study these perpedincular categories. We
will go on to characterise the maps sent to zero by the functor
$\Hom_{\C}(T,-)$ and use this to characterise the maps inverted by this
functor.

Using Lemma~\ref{wakamatsu}, one obtains the following well-known fact
(see~\cite[Prop. 2.3]{iy}).

\begin{lemma}\label{func}
$T^{\perp}$ is covariantly finite, and $^{\perp}T$ is contravariantly finite.
\end{lemma}

The following lemma is a triangulated variation of the Auslander-Reiten
correspondence~\cite[Prop. 1.10]{ar}. This lemma is well-known,
but we include the short proof for convenience.

\begin{lemma}\label{doubleperp}
$^{\perp}(T^{\perp}) = \add T = (^{\perp}T)^{\perp}$.
\end{lemma}

\begin{proof}
We prove the first equality; the proof of the second is dual.
One inclusion is obvious. For the other, let $X$ be in
$^{\perp}(T^{\perp})$, and let $T_0\to X$ be a right $\add
T$-approximation of $X$. Complete it to a triangle $Z \to T_0 \to X \to \Sigma
Z$. Then $Z$ is in $T^{\perp}$ by Lemma~\ref{wakamatsu}, hence by the
assumption on $X$ we have that
$\Hom_{\C}(X,\Sigma T)= 0$,
and the
triangle splits. Therefore $X$ is in $\add T$.
\end{proof}

Consider now the functor $F=\Hom_{\C}(T,-) \colon \C \to
\module \End_{\C}(T)^{\op}$. We first determine which maps are killed and
which maps are inverted by this functor.

\begin{lemma}\label{factors}
Let $f \colon X \to Y$ be a map in $\C$. Then $\Hom_{\C}(T,f) = 0$ if and
only if $f$ factors through $\Sigma T^{\perp}$.
\end{lemma}

\begin{proof}
It is clear that $\Hom_{\C}(T,f) = 0$ if $f$ factors through
$\Sigma T^{\perp}$.

Now assume $\Hom_{\C}(T,f) = 0$. Consider the right $\add T$-approximation
$T_0 \to X$. Complete to a triangle $$U \to T_0 \to X \to \Sigma U,$$
where by Lemma~\ref{wakamatsu} we have that $U$ is in $T^{\perp}$.

Now we have a diagram
$$
\xymatrix{
U \ar[r] & T_0 \ar[r] & X \ar[d]^{f} \ar[r]  &\Sigma U \\
& & Y &
}
$$
where the composition $T_0 \to X \overset{f}{\to} Y$ vanishes by assumption. Hence
the map $f$ factors through $\Sigma U$, which is in $\Sigma T^{\perp}$.
This completes the proof.
%Assume now $\Hom_{\C}(T,f) = 0$. By Serre duality, this implies
%$\Hom_{\C}(f,\Sigma^2T) = 0$.  Let $U \to Y$ be a minimal right $
%\Sigma T^{\perp}$-approximation of $Y$, and complete it to a triangle
%\begin{equation}\label{apptri}
%V \to U
%\to Y \overset{c}{\to} \Sigma V \end{equation}
%in $\C$. By Lemma~\ref{wakamatsu} we have that $V$ belongs to
%$((\Sigma T)^{\perp})^{\perp} = \add  \Sigma T$.
%
%Considering the long exact sequence induced by
%applying $\Hom_{\C}(X, \ )$ to the triangle (\ref{apptri}), we
%have that $$\Hom_{\C}(X,U) \to \Hom_{\C}(X, Y) \to \Hom_{\C}(X, \Sigma V) $$
%is exact. Hence it is sufficient to show that $f$ is in the kernel of
%$$\Hom_{\C}(X, c) \colon\Hom_{\C}(X, Y) \to \Hom_{\C}(X, \Sigma V),$$
%i.e.\ that $cf=0$.
%We have that $\Sigma V$ belongs to $\Sigma^2 T$
%and hence that the map:
%$$\Hom_{\C}(f,\Sigma V) \colon \Hom_{\C}(Y, \Sigma V) \to \Hom_{\C}(X,\Sigma V)$$
%vanishes.
%Hence, indeed  $c \colon Y \to \Sigma V$ is in the kernel of this map,
%i.e.\ $cf=0$. This finishes the proof.
\end{proof}

Let $\widetilde{\S}$ be the class of maps $f \colon X \to
Y$ such that when $f$ is completed to a triangle
$$\Sigma^{-1} Z \overset{h}{\to} X \overset{f}{\to} Y \overset{g}{\to} Z$$
both $g$ and $h$ factor through $\Sigma T^{\perp}$.

\begin{lemma} \label{Stildewelldefined}
The class $\widetilde{\S}$ is well-defined.
\end{lemma}

\begin{proof}
Let $f \colon X\to Y$ be a map in $\C$ and let
$$\Sigma^{-1}Z \overset{h}{\to} X \overset{f}\to Y \overset{g}{\to} Z$$
and
$$\Sigma^{-1}Z' \overset{h'}{\to} X \overset{f}\to Y \overset{g'}{\to} Z$$
be completions of $f$ to a triangle in $\C$.
Since $\C$ is a triangulated category there is a morphism of triangles
$$\xymatrix{
\Sigma^{-1}Z \ar[d]^{\Sigma^{-1}u} \ar[r]^h & X \ar@{=}[d] \ar[r]^f & Y \ar@{=}[d] \ar[r]^g & Z \ar[d]^{u} \\
\Sigma^{-1}Z' \ar[r]^{h'} & X \ar[r]^{f} & Y \ar[r]^{g'} & Z'
}$$
(see e.g. axiom (TR3) in~\cite[1.1]{happel}).
By~\cite[Prop.\ 1.2(c)]{happel}, $u$ (and therefore also $\Sigma^{-1}u$) is
an isomorphism. It follows that $g$ factors through $\X_T$ if and only if
$g'$ factors through $\X_T$, and similarly for $h$ and $h'$.
\end{proof}

We have the following characterisation of $\widetilde{\S}$:

\begin{lemma}\label{invertingmaps}
A map $f \colon X \to Y$ belongs to $\widetilde{\S}$ if and only
if $\Hom_{\C}(T,f)$ is an isomorphism.
\end{lemma}

\begin{proof}
This follows directly by considering the long exact sequence obtained from applying $\Hom_{\C}(T,-)$ to the triangle
$\Sigma^{-1} Z \to X \to Y \to Z$, in combination with Lemma~\ref{factors}.
\end{proof}

Note that the same proof shows:

\begin{lemma} \label{monoepi}
Let $f \colon X\to Y$ and the completed triangle be as above. Then:
\begin{enumerate}
\item[(a)] The map $h$ factors through $\Sigma T^{\perp}$
if and only if $\Hom_{\C}(T,f)$ is a monomorphism.
\item[(b)] The map $g$ factors through $\Sigma T^{\perp}$
if and only if $\Hom_{\C}(T,f)$ is an epimorphism.
\end{enumerate}
\end{lemma}

(Compare with~\cite[Thm. 2.3]{kz}).

\section{Localisation}\label{sl}

In this section, we consider the Gabriel-Zisman localisation of
$\C$ at the class $\widetilde{\S}$ of maps defined above. It turns out,
that in order to study this localisation, it is helpful to consider
localisation at a smaller better-behaved class of maps $\S$ contained
in $\widetilde{\S}$ which has the property that the corresponding
localisation functor also inverts $\widetilde{\S}$. We then investigate
some of the properties of localisation at $\S$ which we shall need to
prove the main result in the next section.

For a class $\M$ of maps in $\C$, the \emph{Gabriel-Zisman
localisation}~\cite[Chapter 1]{gz} $\C_{\M}$ can be defined as follows.
Let the objects in $\C_{\M}$
be the same as the objects in $\C$.
The maps in $\C_{\M}$ are defined as follows (following~\cite[Section 2.2]{krause}).
For each element $m$ in $\M$, introduce an element $x_m$.
This is the formal inverse of $m$.
Then construct an oriented graph $\G$ as follows. The vertices of
$\G$ are the objects of $\C$ and the arrows are the maps in $\C$ together with
the elements $x_m$ for each $m$ in $\M$. Here the orientation of
the arrow corresponding to a map from $X$ to $Y$ is $X \to Y$,
while for an element $m$ in $\M$,
the edge $x_m$ has the same vertices as $m$, but the opposite orientation.
Then the maps in $\C_{\M}$ from $X$ to $Y$ are equivalence
classes of paths in $\G$ starting at $X$ and ending
at $Y$.
The equivalence relation is defined as follows. Consider the
relation given by the following:
\begin{itemize}
\item[-] two consecutive arrows can be replaced by their composition;
\item[-] for $m$ in $\M$, a composition $X \overset{m}{\to} Y
  \overset{x_m}{\to} X$ or a composition  $X \overset{x_m}{\to} Y
  \overset{m}{\to} X$ can be replaced by $X \overset{1_X}{\to} X$
\end{itemize}
and close under reflexivity, symmetry and transitivity to obtain an
equivalence relation.

Note that we have assumed $\C$ to be skeletally small, in order to
avoid set-theoretic problems in forming localisations.

There is a canonical functor
$L_{\M} \colon \C \to \C_{\M}$ with the property that each functor starting in
$\C$ which inverts the maps in $\M$ factors uniquely through $L_{\M}$.

In our situation, by Lemma~\ref{invertingmaps}, this means that there is a unique functor
$G \colon \C_{\widetilde{\S}} \to \module \End_{\C}(T)^{\op}$, making the diagram
$$
\xymatrix{
\C  \ar[dr]_{L_{\widetilde{\S}}} \ar[rr]^(0.35){H =\Hom_{\C}(T,-)}  &  &
\module \End_{\C}(T)^{\op} \\
& \C_{\widetilde{\S}} \ar[ur]_G   &
}
$$
commute, and our main theorem is that this functor
is actually an equivalence. 
However, in order to prove this, it will be convenient to
consider a subclass $\S$ of $\widetilde{\S}$ consisting of maps $f \colon X \to
Y$ such that when $f$ is completed to a triangle
$$\Sigma^{-1}Z \overset{h}{\to} X \overset{f}{\to} Y \overset{g}{\to} Z,$$
$g$ factors through $\Sigma T^{\perp}$, while $\Sigma^{-1}Z$ is in $\Sigma
T^{\perp}$. Note that this class is also well-defined, using a similar
argument to that used for $\widetilde{\S}$ in Lemma~\ref{Stildewelldefined}
and noting that $\X_T$ is closed under isomorphism.
Since $\S$ is contained in $\widetilde{\S}$, it is clear that
$\Hom_{\C}(T,f)$ is an isomorphism for
any map $f \colon X \to Y$ in $\S$, and we get the following diagram:
\begin{equation}
\label{diag}
\xymatrix{
\C \ar[dr]^{L_{\S}} \ar[ddr]_{L_{\widetilde{\S}}} \ar[rr]^(0.35){H= \Hom_{\C}(T,-)}  &  &
\module \End_{\C}(T)^{\op} \\
& \C_{\S} \ar[ur]^{F} \ar[d]^{J}& \\
& \C_{\widetilde{\S}} \ar[uur]_{G}   &
}
\end{equation}
We claim that the diagram commutes. By the universal property of localisation, we have 
$FL_{\S} = H$ and $J L_{\S} = L_{\widetilde{\S}}$. Using in addition that
$G L_{\widetilde{\S}} = H$, it follows that
$G J L_{\S} = F L_{\S}$. 
By the universal property of $L_{\S}$, we have 
$G J  = F $.

Our strategy is to first prove that the functor $F \colon \C_{\S} \to
\module \End_{\C}(T)^{\op}$ is an equivalence. Then, using that both localisation functors
invert exactly the maps in $\widetilde{\S}$, the functor
$\C_{\S} \to \C_{\widetilde{\S}} $ is an equivalence, and our result
follows.

\begin{remark}
The classes of maps $\S$ and $\widetilde{\S}$ do not satisfy
all of the axioms for admitting a calculus of left fractions or a calculus
of right fractions. It is easy to see that the axiom LF3 (and its
counterpart RF3), in the notation of~\cite[Section 3.1]{krause},
are not in general satisfied.
Note that this implies that the corresponding localisation functors
do not admit a right or left adjoint; see \cite[Section 2.3]{krause}.
\end{remark}

We consider the full subcategory
$\C(T)$ of $\C$ consisting of objects
$X$ in $\C$ such that there exists a triangle
$T_1 \to T_0 \to X \to \Sigma T_1$ in $\C$, with $T_0,T_1$ in $\add T$
(see~\cite[Prop. 6.2]{iy},~\cite[Section 5.1]{kr}).
The following characterises the objects in $\C(T)$.

\begin{lemma}\label{ct}
For an object $X$ in $\C$, the following are equivalent.
\begin{itemize}
\item[(a)] $X$ is in $\C(T)$.
\item[(b)] If, in the triangle  $U \to T_0 \overset{f}{\to} X \to
  \Sigma U$, the map $f$ is a right  $\add T$-approximation, then $U$ is also in $\add T$.
\item[(c)]  If, in the triangle  $U \to T_0 \overset{f}{\to} X \to
 \Sigma U$, the map $f$ is a minimal right  $\add T$-approximation, then $U$ is also in $\add T$.
\end{itemize}
\end{lemma}

\begin{proof}
Assume there is a
triangle
$T_1 \to T_0 \to X \to \Sigma T_1$ in $\C$, with $T_0,T_1$ in $\add
T$. Then, using the fact
that $\Hom_{\C}(T,-)$is a homological functor and that $T$ is rigid, it
is clear that
$T_0 \to X$ is a right $\add T$-approximation.
The statement now follows from combining this with Lemma~\ref{app-lemma}.
\end{proof}

The importance of $\C(T)$ here is due to the following lemma.

\begin{lemma}\label{identify}
Let $Y$ be an object in $\C$. Then there exists a map $f \colon X\to Y$ in $\S$
where $X$ is an object in $\C(T)$.
\end{lemma}

\begin{proof}
Let $T_0 \overset{u}{\to} Y$ be a minimal right $\add T$-approximation
of $Y$ and complete to a triangle
$$Z \to T_0 \overset{u}{\to} Y \to \Sigma Z.$$
Let $T_1 \to Z$ be the minimal right $\add T$-approximation of $Z$,
and complete to a triangle
$$\Sigma^{-1} U \to T_1 \to Z \to U.$$
We have that both $Z$ and $\Sigma^{-1} U$ belong to $T^{\perp}$, by Lemma~\ref{wakamatsu}.

Consider now the following diagram, obtained by applying the octahedral
axiom to the composition $T_1 \to Z \to T_0$.
$$
\xymatrix{
\Sigma^{-1} U \ar[r] \ar[d] & T_1  \ar[r]^{w} \ar@{=}[d] & Z \ar[r] \ar[d]^{v}   & U \ar[d] \\
\Sigma^{-1} X \ar[r] & T_1  \ar[r]^{vw} & T_0  \ar[r] \ar[d]^{u} & X \ar[d]^{s} \\
& & Y  \ar@{=}[r] \ar[d] & Y  \ar[d]^{t}\\
& & \Sigma Z  \ar[r] & \Sigma U
}
$$
The map $X \overset{s}{\to} Y$ belongs to $\S$, since $U$ is in
$\Sigma T^{\perp}$ and the map $Y \overset{t}{\to} \Sigma U$ factors
through $\Sigma Z$, which belongs to $\Sigma
T^{\perp}$.
\end{proof}

We consider now the category $\C_{\S}$  obtained by localising $\C$ with respect to
$\S$. For a map $f$ in $\C$, we denote its image in $\C_{\S}$ by
$\underline{f} = L_{\S}(f)$.
We provide a series of lemmas needed for the proof of our main theorem.

First note that we have the following direct consequence of Lemma~\ref{identify}.
\begin{corollary}\label{objects}
Let $Y$ be an object of $\C_{\S}$.  Then there is an object $X$ of
$\C(T)$ such that $X \simeq Y$ in $\C_{\S}$.
\end{corollary}

Note that a priori we do not know if $\C_{\S}$ is additive. This
will follow from our main result. We will need some elementary properties of $\C_{\S}$
and the localisation functor $L_{\S} \colon \C \to \C_{\S}$.

\begin{lemma}\label{elementary}
\begin{itemize}
\item[(a)] For $U$ in $\Sigma T^{\perp}$, the zero map $u_0 \colon U \to 0$
  belongs to
  $\S$. Furthermore, the inverse of $\underline{u_0}$ in $\C_{\S}$ is
  $\underline{u^0}$, where $u^0 \colon 0 \to U$ is the zero map.
\item[(b)] For $U$ in $\Sigma T^{\perp}$, and any object $X$ in $\C$,
the projection $\pi_X \colon X \amalg U \to X$ belongs to
  $\S$. Furthermore, the inverse of $\underline{\pi_X}$ in $\C_{\S}$ is
  $\underline{\iota_X}$, where $\iota_X \colon X \to X \amalg U$ is the inclusion map.
\item[(c)] Let $u \colon X \to Y$ be a map in $\C$ factoring through $\Sigma
T^{\perp}$. Then $\underline{u} = \underline{0}$ in $\C_{\S}$.
\item[(d)] Let $u,v \colon X \to Y$ be maps in $\C$, such that $v$
  factors through $\Sigma T^{\perp}$. Then $\underline{u+v} = \underline{u}$ in $\C_{\S}$.
\end{itemize}
\end{lemma}

\begin{proof}
Part (a) is straightforward. For (b)
consider the triangle
$$U \to U \amalg X  \overset{\pi_X}{\to }  X \overset{0}{\to}  \Sigma U.$$
Since $U$ is in  $\Sigma T^{\perp}$, we have that $\pi_X$ belongs to
$\S$.
Note that as $\pi_X \iota_X = 1_X$ in $\C$,
we also have that $\underline{\pi_X} \underline{\iota_X} = \underline{1_X}$ in
$\C_{\S}$. This proves the second claim.
Part (c) is a direct consequence of (a). To prove (d), assume $v$ factors
as $X \overset{f}{\to} U \overset{g}{\to} Y$, with $U$ in $\Sigma T^{\perp}$.
Then, using (b), we obtain \begin{multline*} \underline{u+v} =
  \underline{u + gf} =
\underline{\left( \begin{smallmatrix} g  & 1_Y \end{smallmatrix} \right)}
\underline{\left( \begin{smallmatrix} f \\  u \end{smallmatrix}
  \right)} =
\underline{\left( \begin{smallmatrix} g  & 1_Y \end{smallmatrix} \right)}
\underline{\left( \begin{smallmatrix} 1_X &  0 \\ 0 & u \end{smallmatrix}
  \right)}
 \underline{\left( \begin{smallmatrix} f  \\
        1_X \end{smallmatrix} \right)}
= \\
\underline{\left( \begin{smallmatrix} g & 1_Y \end{smallmatrix}
  \right)}
\underline{\iota_Y}
\underline{\pi_Y}
\underline{\left( \begin{smallmatrix} 1_U  & 0\\ 0 & u \end{smallmatrix} \right)}
\underline{\iota_X}
\underline{\pi_X}
\underline{\left( \begin{smallmatrix} f\\ 1_X \end{smallmatrix}
  \right)} =
\underline{1_Y}
\underline{\left( \begin{smallmatrix} 0  & 1_Y  \end{smallmatrix} \right)}
\underline{\left( \begin{smallmatrix} 1_U & 0\\ 0 & u \end{smallmatrix}
  \right)}
\underline{\left( \begin{smallmatrix} 0\\ 1_X \end{smallmatrix}
  \right)}
\underline{1_X}
=
\underline{u}, \end{multline*}
where (c) is used for the second-to-last equality.
\end{proof}

The next lemma will be helpful in simplifying the description of maps in $\C_{\S}$.
We would like to thank Yann Palu for this simplification of
an earlier version of the proof of the lemma.

\begin{lemma}\label{factoring}
Let $U$ be an object in $\C(T)$, let $u \colon U \to Y$ be a map
in $\C$ and let $s \colon X\to Y$ be a map in $\S$. Then $u$ factors
through $s$.
\end{lemma}

\begin{proof}
Since $U$ lies in $\C(T)$, there is a triangle:
$$T_1 \to T_0 \to U \to \Sigma T_1$$
with $T_0,T_1$ in $\add T$.
Complete $s$ to a triangle $\Sigma^{-1}Z \overset{h}{\to} X
\overset{s}{\to} Y \overset{g}{\to} Z$
where $\Sigma^{-1}Z$ lies in $\Sigma T^{\perp}$ and
$g$ factors through $\Sigma T^{\perp}$.
We thus have the diagram:
$$\xymatrix{
T_1 \ar[r] \ar@{.>}[d] & T_0 \ar^f[r] \ar@{.>}[d] & U \ar[r] \ar^u[d] & \Sigma T_1 \ar^v@{.>}[d] \\
\Sigma^{-1}Z \ar^h[r] & X \ar^s[r] & Y \ar^g[r] & Z
}$$
Since $T_0$ lies in $\add T$ and $g$ factors through $\Sigma T^{\perp}$,
$guf=0$ and there are maps $T_1\to \Sigma^{-1} Z$ and $T_0\to X$ giving a map
of triangles as indicated in the diagram. Since $T_1$ lies in $\add T$ and $Z$ lies in $\Sigma^2 T^{\perp}$,
the map $v$ vanishes, so $gu=0$ and $u$ factors through $s$ as required.
\end{proof}

%and hence so does the composition
%$$U \overset{u}{\to} Y \overset{g}{\to} Z.$$
%Therefore $gu$ factors through the
%minimal left $\Sigma T^{\perp}$-approximation $U \overset{t}{\to} V$ of $U$, and
%hence there is a commutative square:
%$$
%\xymatrix{
% U \ar[d]^{u} \ar[r]^{t} & V \ar[d]^{v}   \\
% Y \ar[r]^{g}   & Z  &
%}
%$$
%Complete the map $t$ to a triangle
%$$\Sigma^{-1}V  \to T_0 \overset{w}{\to} U \overset{t}{\to} V.$$
%Using Lemma \ref{wakamatsu}, we have that $T_0$ is in
%${}^{\perp}(T^{\perp})  = \add T$, and that $w \colon T_0 \to U$ is a right $\add %T$-approximation.
%The above commutative square gives rise to the following commutative diagram,
%where the rows are triangles:
%$$
%\xymatrix{
%\Sigma^{-1} V \ar[r] \ar[d]  & T_0 \ar[r]^{w} \ar[d]  &   U
%\ar[r]^{t} \ar[d]^{u} & V \ar[d]^{v}  \ar[r]  &  \Sigma T_0 \\
%\Sigma^{-1} Z \ar[r]  &  X  \ar[r]^{s}   & Y \ar[r]^{g}  & Z &
%}
%$$
%Using our assumption on $U$ and Lemma~\ref{ct}, we have that
%$\Sigma^{-1} V$ is in $\add T$, and hence $V$ belongs to $\add \Sigma
%T$. But this implies that the map $v \colon V \to Z$ vanishes, since
%$Z$ is in $\Sigma^2 T^{\perp}$.  Therefore the composition $gu$ also
%vanishes, which implies that $u$ factors through $s$. This finishes
%the proof.
%\end{proof}

As a consequence of this we obtain the following.

\begin{proposition}\label{surjection}
Let $U,V$ be objects in $\C$ with $U$ in $\C(T)$. Then $\Hom_{\C}(U,V)
\to \Hom_{\C_{\S}}(U,V)$ is surjective.
\end{proposition}

\begin{proof}
A map in $\C_{\S}$ is a composition of maps in $\C$ and formal
inverses of maps in $\S$. Assume we have a composition
$U \overset{\underline{u}}{\to}Y \overset{\underline{s}^{-1}}{\to} X$
for a map $X \overset{s}{\to} Y$
in $\S$. Then, in $\C$ we have the commutative diagram
$$
\xymatrix{
 & U \ar[d]^{u} \ar@{->}[dl]_{h}   \\
 X\ar[r]^{s}   & Y &
}
$$
where the map $h\colon U {\to} Y$ such that $u=sh$ exists by Lemma~\ref{factoring}.

Now, $\underline{s}^{-1}\underline{u} =
\underline{s}^{-1}\underline{s} \underline{h} = \underline{h}$.

By this it is clear that any map in $\C_{\S}$ from the object $U$ of
$\C(T)$ can be obtained as the
image of a map in $\C$.
\end{proof}

%Let $X,U$ be objects in $\C$, with $U$ in $\Sigma T^{\perp}$.
%Let $\iota_X \colon X \to X \amalg U$ be the natural embedding and let
%$\pi_X \colon X \amalg U \to X$ be the natural projection.
%Then $\pi_X$ belongs to $\S$, and the formal inverse of $\underline{\pi_X}$ in $\C_S$
%coincides with $\underline{\iota_X}$.
%\end{lemma}

%\begin{proof}
%Consider the triangle
%$$U \to U \amalg X  \overset{\pi_X}{\to }  X \overset{0}{\to}  \Sigma U.$$
%Since $U$ is in  $\Sigma T^{\perp}$, we have that $\pi_X$ belongs to
%$\S$.
%
%Note that as $\pi_X \iota_X = 1_X$ in $\C$,
%we also have that $\underline{\pi_X} \underline{\iota_X} = 1_X$ in
%$\C_{\S}$. This proves the second claim.
%\end{proof}

%\begin{lemma}
%Let $u,v \colon X \to Y$ be maps in $\C$, and suppose $v$ factors
%through $\Sigma T^{\perp}$ in$\C$.
%Then $\underline{u+v} = \underline{u}$ in $\C_{\S}$.
%\end{lemma}

\section{Main result}

As before, let $T$ be a rigid object in the triangulated Hom-finite
Krull-Schmidt category $\C$, let $\S$ be as described in
Section~\ref{sl} and let $\widetilde{\S}$ be as defined in
Section~\ref{section-rigid}.
In this section we prove our main result, i.e.\ that
if $\C$ is skeletally small, the localisation of $\C$ at the class
$\widetilde{S}$ is equivalent to $\module \End_{\C}(T)$.

We first consider the canonical functor
$F \colon \C_{\S} \to \module \End_{\C}(T)^{\op}$, as described in
Section \ref{sl}, and prove that it is an equivalence.

\begin{proposition}\label{firstlocal}
The functor $F \colon \C_{\S} \to \module \End_{\C}(T)^{\op}$
is an equivalence.
 \end{proposition}

We recall the following result:

\begin{theorem} [Iyama-Yoshino] \label{Tblanktheorem1}
Let $T$ be a rigid object in a Hom-finite Krull-Schmidt triangulated
category $\C$. Then the functor $\Hom_{\C}(T,-)$ induces an equivalence
$$\C(T)/\Sigma T \to \module \End_{\C}(T)^{\op}.$$
\end{theorem}

This was proved in~\cite[Prop. 6.2]{iy}. The case where
$\C$ is $d$-Calabi-Yau was proved in~\cite[Section 5.1]{kr}.
In view of Lemma~\ref{factors}, we will need the following variation of
Theorem~\ref{Tblanktheorem1}; the proof is fairly similar.
For an additive category $\C$ and additive subcategory $\C'$, we will use $\C/\C'$ to
denote the quotient of $\C$ by the ideal of maps which factor through $\C'$; it is
also an additive category. For an object $C$ in $\C$, we let
  $\C/C$ denote $\C/\add C$.

\begin{lemma}\label{equiv}
The functor  $H= \Hom_{\C}(T,-)  \colon \C \to \module \End_{\C}(T)^{\op}$
is dense. Its restriction $H'$ to $\C(T)$ is full and
induces an equivalence $\C(T)/ \Sigma T^{\perp} \to \module \End_{\C}(T)^{\op}$.
\end{lemma}

\begin{proof}
Let $M$ be a finite dimensional $\End_{\C}(T)^{\op}$-module and
choose a minimal projective presentation $$P_1 \to P_0 \to M \to 0.$$
Since $\Hom_{\C}(T,-)$ induces an equivalence from $\add T$ to the
finite dimensional projective $\End_{\C}(T)^{\op}$-modules, the map
$P_1 \to P_0$ is the image of a map $T_1 \to T_0$ in $\add T$.
Complete to a triangle
$$T_1 \to T_0 \to X \to \Sigma T_1.$$
Since $\Hom_{\C}(T,\Sigma T_1)= 0$, we get an exact sequence
$$\Hom_{\C}(T,T_1) \to \Hom_{\C}(T,T_0) \to \Hom_{\C}(T,X) \to 0.$$
Hence, it is clear that $M \simeq \Hom_{\C}(T,X)$ and thus that $H$
is dense.

Now, let $X,Y$ be objects in $\C(T)$, and consider a map $\alpha
\colon \Hom_{\C}(T,X)  \to \Hom_{\C}(T,Y)$.
Then there are triangles $T_1 \to T_0 \to X \to \Sigma T_1$ and
$U_1 \to U_0 \to Y \to \Sigma U_1$ whose images under $\Hom_{\C}(T,-)$
give minimal projective presentations for $\Hom_{\C}(T,X)$ and
$\Hom_{\C}(T,Y)$, respectively.

Hence there are vertical maps such that the following diagram
commutes:
$$
\xymatrix{
\Hom_{\C}(T,T_1) \ar[r] \ar[d] & \Hom_{\C}(T,T_0) \ar[r] \ar[d] &
\Hom_{\C}(T,X) \ar[r] \ar[d]^{\alpha} &
0 \\
\Hom_{\C}(T,U_1) \ar[r] & \Hom_{\C}(T,U_0) \ar[r] &\Hom_{\C}(T,Y) \ar[r]
& 0
}
$$
Lifting the left hand square of this diagram to $\add T$, we obtain a
diagram
$$
\xymatrix{
T_1 \ar[r] \ar[d] & T_0 \ar[r] \ar[d] &  X \ar@{.>}[d] \ar[r] & \Sigma
T_1 \ar[d]
 \\
U_1  \ar[r]  & U_0 \ar[r] & Y \ar[r] & \Sigma U_1
}
$$
The induced map $X \to Y$ is mapped to $\alpha$ and hence the functor
$H'$ is full. The statement now follows from Lemma~\ref{factors}.
\end{proof}

\begin{proof}[Proof of Proposition~\ref{firstlocal}.]
First note that for an object $X$ in $\C_{\S}$ we have that $F(X) = H(X) =
\Hom_{\C}(T,X)$. If $f\colon X \to Y$ is a map in $\C$, we have that $F(\underline{f})
= \Hom_{\C}(T,f)$.

We have that $F$ is dense by Lemma~\ref{equiv} and the fact that $H = F L_{\S}$.
We now show that that $F$ is full and faithful. For this
let $X,Y$ be objects in $\C_{\S}$. By Lemma~\ref{identify},
there are objects $X',Y'$ in $\C(T)$, and maps $u \colon X'
\to X$ and $v \colon Y' \to Y$ in $\S$, such that
 $\underline{u} \colon X'
\to X$ and $\underline{v} \colon Y' \to Y$ are isomorphisms.
Let $\beta = F(\underline{u})$, and let $\gamma = F(\underline{v})$.

\noindent {\bf $F$ is full:}

\noindent Let $\alpha \colon FX \to FY$ be an arbitrary map in
$\module \End_{\C}(T)^{\op}$,
and let $\alpha' = \gamma^{-1} \alpha \beta$, i.e.\ we have the
commutative diagram

$$
\xymatrix{
FX'  \ar[r]^{\beta} \ar@{->}[d]^{\alpha'} & FX \ar[d]^{\alpha} \\
FY' \ar[r]^{\gamma} & FY
}
$$

By Lemma~\ref{equiv} $H'$ is dense, so, since $FL_{\S} = H$ and $X'$ lies
in $\C(T)$, there is a map $f' \colon X' \to Y'$ in $\C$ such
that $FL_{\S}(f')= \alpha'$. Hence $F(\underline{f'}) = \alpha'$.

Setting $f = \underline{v}  \underline{f'} \underline{u}^{-1} \colon X \to
Y$ in $\C_{\S}$, we obtain
$F(f) =  F(\underline{v})  F(\underline{f'}) F(\underline{u})^{-1} =
\gamma \alpha' \beta^{-1} = \gamma \gamma^{-1} \alpha' \beta
\beta^{-1} = \alpha$ as required.

\noindent {\bf $F$ is faithful:}

\noindent Let $X,Y$ be objects of $\C_{\S}$, and assume we have maps $f,g\colon
X \to Y$ in $\C_{\S}$ such that $F(f)= F(g)$.
By Proposition~\ref{surjection}, there are $f', g'\colon X'\to Y'$ in
$\C$ such that $\underline{f'} =
\underline{v}^{-1} f \underline{u}$ and $\underline{g'} =
\underline{v}^{-1} g \underline{u}$.

Then \begin{multline*}
F(\underline{f'}) =
F(\underline{v}^{-1} f \underline{u}) =
F(\underline{v})^{-1} F(f) F(\underline{u}) = \\
F(\underline{v})^{-1} F(g) F(\underline{u}) =
F(\underline{v}^{-1} g \underline{u}) = F(\underline{g'}).
\end{multline*}
Hence we have that $H(f') = H(g')$, and by Lemma~\ref{equiv}, this implies
that $g'-f'$ factors through $\Sigma T^{\perp}$ in $\C$.
By Lemma~\ref{elementary} we have that $$\underline{f'} = \underline{f' +
  (g'-f')} = \underline{g'}$$
in $\C_{\S}$, and hence
$f= \underline{v} \underline{f'} \underline{u}^{-1} = \underline{v}
\underline{g'} \underline{u}^{-1}  = g$ in $\C_{\S}$ as required.

Hence $F$ is dense, full and faithful, and we are done.
\end{proof}

By Lemma~\ref{invertingmaps}, we have that $\Hom_{\C}(T,f)$ is invertible if and
only if $f$ belongs to $\widetilde{\S}$.
Since $FL_S =H$ (see Diagram~\eqref{diag}) and $F$
is an equivalence (by Proposition~\ref{firstlocal}), it follows that
$L_S$ inverts the maps in $\widetilde{\S}$. Hence there is an induced
functor $I \colon C_{\widetilde{\S}}\to C_{\S}$ such that
$IL_{\widetilde{S}} = L_{\S}$.
Since $L_{\widetilde{\S}}$ inverts the maps in $\S$ there
is an induced functor $J \colon C_{\S}\to C_{\widetilde{\S}}$ such that
$JL_{\S} =L_{\widetilde{\S}}$. It follows that $IJL_{\S}= L_{\S}$. It
follows from the universal property of $L_{\S}$ that $IJ$
is equal to the identity functor on $C_{\S}$. Similarly $JI$ is equal
to the identity functor on $C_{\widetilde{\S}}$. Hence the induced
functors $I$ and $J$ between $C_{\S}$ and $C_{\widetilde{S}}$ 
are isomorphisms. Our main theorem follows:

%By the universal property of
%localisation it follows that the induced functor $\C_{\S} \to
%\C_{\widetilde{\S}}$ is an equivalence (see Diagram~\eqref{diag}).
%Hence, we have proved our main theorem.

\begin{theorem}\label{main}
Let $\C$ be a skeletally small Hom-finite Krull-Schmidt triangulated category with rigid object $T$.
Let $\widetilde{\S}$ be the class of maps defined above. Then the induced functor
$G \colon \C_{\widetilde{\S}} \to \module \End_{C}(T)^{\op}$ is an equivalence.
\end{theorem}

\section{The cluster-tilting object case} \label{s:clustertiltingcase}

In this section we compare our results with recent work showing how
to obtain module categories of (opposite) endomorphism algebras of rigid
objects in triangulated categories from the categories themselves.

We call a rigid object $T$ in $\C$ a \emph{cluster-tilting
object} if $\Ext^1_{\C}(T,X)  = 0$ only if $X$ is in $\add T$. Note
that by~\cite[Lemma 3.2]{kz} (or, alternatively, Lemma~\ref{doubleperp}),
it also follows that $\Ext^1_{\C}(X,T)  = 0$ only if $X$ is in $\add T$,
when $T$ is a cluster-tilting object (see also~\cite[Defn. 3.1]{kz}),
so the cluster-tilting objects studied here coincide with the maximal
$1$-orthogonal objects studied in \cite{iy, kz}.
The motivation for our work was the following theorem, which generalises previous
work dealing with Calabi-Yau triangulated categories \cite[Prop. 2.1]{kr} or cluster categories~\cite[Thm. 2.2]{bmr} (note that part (a) was recalled as Theorem~\ref{Tblanktheorem1} above but we repeat it here for comparison).
Also note that we state results here only as they apply to the object
case (rather than for a rigid subcategory).

\begin{theorem} \label{Tblanktheorem2}
Let $T$ be a rigid object in a Hom-finite Krull-Schmidt triangulated
category $\C$.
\begin{itemize}
\item[(a)] [Iyama-Yoshino] The functor $\Hom_{\C}(T,-)$ induces an equivalence
$$\C(T)/ \Sigma T \to \module \End_{\C}(T)^{\op}.$$
\item[(b)] [Koenig-Zhu] If $T$ is a cluster-tilting object then there is
an equivalence
$$\C / \add \Sigma T \to \module \End_{\C}(T)^{\op}.$$
\end{itemize}
\end{theorem}

As mentioned above, part (a) was proved in~\cite[Prop. 6.2]{iy} (the $d$-Calabi Yau case in~\cite[Section 5.1]{kr}). Part (b) was proved in~\cite[Cor. 4.4]{kz}.

Note that to prove (b) from (a), it is sufficient to realise that
$\C(T) = \C$ in the case where $T$ is cluster-tilting. This is well-known; we
repeat the easy proof here for convenience.

\begin{lemma}
If $T$ is a cluster-tilting object in $\C$, then $\C(T) =\C$.
\end{lemma}

\begin{proof}
Let $C$ be in $\C$ and consider the triangle
$$U \to T_0 \to C \to \Sigma U,$$
obtained by completing
a right $\add T$-approximation $T_0 \to C$.
Then $U$ lies in $T^{\perp} = \add T$ by Lemma \ref{wakamatsu}.
\end{proof}

%If $T$ is a cluster-tilting object in $\C$, then $\C = \C(T)$, by \cite{kr}. Hence the
%theorem by Iyama-Yoshino generalises the theorem by
%Keller-Reiten \cite{kr}, saying that
%$\C/ \Sigma T$ is equivalent to $\module \End_{\C}(T)^{\op}$, when $T$ is a
%tilting object in $\C$.

In this section, we point out how Theorem \ref{Tblanktheorem2} relates to
our main theorem.

%\begin{lemma}\label{mapsequal}
%Let $X,Y$ be objects in $\C(T)$ and $u,v \colon X \to Y$ be maps in
%$\C$.
%Then $\underline{u} = \underline{v}$ in $\C_{\S}$ if and only if $v-u$ factors
%through $\Sigma T^{\perp} $ in $\C$.
% \end{lemma}
%
%\begin{proof}
%Suppose $w=v-u$ factors through $\Sigma T^{\perp} $ in $\C$. Then, by
%lemma \ref{}, $\underline{u} = \underline{u +w} = \underline{u+ (v-u)}
%= \underline{v}$.
%
%Conversely, if $\underline{u} = \underline{v}$ in $\C_{\S}$, then
%$\Hom_{\C}(T,u-v) = 0$, which implies that $u-v$ factors through
%$\Sigma T^{\perp}$ by lemma \ref{factors}.
%\end{proof}

%\begin{proof}
%This follows from
%proposition \ref{surjection} and lemma \ref{mapsequal}.
%\end{proof}

We first give a description of the maps with domain in $\C(T)$
which lie in the kernel of $\Hom_{\C}(T,-)$. We would like to thank
Yann Palu for his short proof of this lemma which replaces the
longer version in an earlier version of the paper.

\begin{lemma}\label{smallerkernel}
Let $X,Y$ be objects in $\C$, with $X$ in $\C(T)$. Suppose that $f
\colon X \to Y$ factors through $\Sigma T^{\perp}$. Then $f$ factors
through $\add \Sigma T$.
\end{lemma}

\begin{proof}
Since $X$ lies in $\C(T)$, there is a triangle:
$$T_1\rightarrow T_0 \overset{g}{\rightarrow} X \overset{h}{\rightarrow} \Sigma T_1,$$
with $T_0$, $T_1$ in $\add T$. Since $f\colon X \to Y$ factors through
$\Sigma T^{\perp}$ and $T_0$ lies in $\add T$, $fg=0$, so $f$ factors through
$h$ and hence through $\add \Sigma T$, since $T_1$ lies in $\add T$.
\end{proof}
%Let $\Sigma T_0 \to Y$ be  a minimal right $\add \Sigma T$-approximation,
%and complete it to a  triangle
%$$\Sigma^{-1}Z \to \Sigma T_0 \to Y \to Z$$
%in $\C$. Let $X \to U$ be the minimal left $\Sigma
%T^{\perp}$-approximation of $X$.
%
%Since the map $f \colon X \to Y$ factors through
% $\Sigma T^{\perp}$, we get a commutative diagram:
%
%$$
%\xymatrix{
%\Sigma^{-1}U \ar[r] \ar[d] & V \ar[r] \ar[d] & X \ar[r] \ar[d]_f & U
%\ar[r] \ar[d] & \Sigma V  \ar[d] \\
%\Sigma^{-1} Z \ar[r] & \Sigma T_0 \ar[r] & Y \ar[r] & Z \ar[r]  & \Sigma^2 T_0
%}
%$$
%where the rows are triangles, and
%where $\Sigma^{-1} Z$ belongs to $\Sigma T^{\perp}$ by
%Lemma~\ref{wakamatsu}. It follows that $Z$ is in $\Sigma^2 T^{\perp}$.
%By Lemma~\ref{wakamatsu}, the map $V \to X$ is a right
%$\add T$-approximation of $X$, and hence $\Sigma^{-1}U$ is in $\add T$, and $U$
%is in $\add \Sigma T$. It follows that the map $U \to Z$ is the zero map,
%and hence that $f$ factors through $\Sigma T_0 \to Y $ as required.
%\end{proof}

Combining Lemmas~\ref{smallerkernel} and~\ref{equiv}, we obtain part (a) of
Theorem \ref{Tblanktheorem2}. Part (b) follows from the observation that $\C(T) =
\C$, when $T$ is a cluster-tilting object.

Summarising, we have the following:

\begin{theorem}
There are equivalences of categories
$$\C(T)/ \Sigma T \to \C(T)/ \Sigma T^{\perp}  \to
\module \End_{\C}(T)^{\op}  \to \C_{\widetilde{\S}}$$
\end{theorem}

\begin{proof}
The first equivalence follows from Lemma~\ref{smallerkernel}, the second is
Lemma~\ref{equiv}, while the third is our main result, Theorem~\ref{main}.
\end{proof}

Assume $T$ is a cluster-tilting object in $\C$, so that we have $\C = \C(T)$. Let us finish by pointing out
that the equivalence obtained by the composition
$\C/ \Sigma T^{\perp}  \to
\module \End_{\C}(T)^{\op}  \to \C_{\widetilde{\S}}$ in the above
theorem, has a natural interpretation.

Let $Q \colon \C \to \C / \Sigma T^{\perp}$ be the canonical quotient
functor, and $H_0 \colon \C / \Sigma T^{\perp} \to \module \End_{\C}(T)^{\op}$
the induced functor. We
then have $H_0 Q = H =  G L_{\widetilde{\S}}$.
It is clear that $Q$ maps $\Sigma T^{\perp} $ to $0$, and is universal
among additive functors with respect to this property.
Moreover, by Lemma~\ref{invertingmaps} in combination with the
fact that $H_0$ is an equivalence, it is clear that $Q$ inverts the
maps in $\widetilde{\S}$.

On the other hand, $L_{\widetilde{\S}}$ is universal with respect to
inverting $\widetilde{\S}$, and it is clear that $L_{\widetilde{\S}}$
maps  $\Sigma T^{\perp} $ to $0$. Moreover, $L_{\widetilde{\S}}$ is additive
with respect to the additive structure on $\C_{\widetilde{\S}}$ induced
by the equivalence $G$.

Combining these two facts, we obtain canonical functors
$U \colon  \C/\Sigma T^{\perp} \to \C_{\widetilde{\S}}$, and
$V \colon \C_{\widetilde{\S}} \to  \C/\Sigma T^{\perp} $. Arguing as
for diagram \eqref{diag}, and using that $Q$ is full, we obtain a
commutative diagram

$$
\xymatrix{
\C \ar[dr]^{Q} \ar[dddr]_{L_{\widetilde{\S}}} \ar[rr]^(0.4){H= \Hom_{\C}(T,-)}  &  &
\module \End_{\C}(T)^{\op} \\
& \C/ \Sigma T^{\perp} \ar@<0.4ex>[dd]^U \ar[ur]^{H_0}  & \\
& & \\
& \C_{\widetilde{\S}} \ar@<0.4ex>[uu]^V \ar[uuur]^{G}   &
}
$$

By universality, it follows that $U$ is an isomorphism and that
$V$ is the inverse of $U$.

%Using that $H_0 Q = G L_{\widetilde{\S}}$, it is clear that $U = $ is
%naturally equivalent to $G^{-1} H_0 $, and that $V$ is naturally
%equivalent to $H_0^{-1} G$, where $G^{-1}$ and $H_0^{-1}$ are quasi-inverses
%of $G$ and $H_0$, respectively.

\section{Cotorsion pairs}

Nakaoka~\cite{nakaoka} considers the notion of a \emph{cotorsion pair}
in a triangulated category. In this section we compare the results obtained here
to those of Nakaoka.

According to~\cite[2.3]{nakaoka} a cotorsion pair can be
defined as a pair $(\U,\V)$ of full additive subcategories satisfying
\begin{enumerate}
\item[(a)] $U^{\perp}=V$;
\item[(b)] ${}^{\perp}V=U$;
\item[(c)] For any object $C$, there is a (not necessarily unique) triangle:
$$U \to C \to \Sigma V \to \Sigma U,$$
with $U\in \U$ and $V\in \V$.
\end{enumerate}

Nakaoka points out that $(\U,\V)$ is a cotorsion pair if and only
if $(\U,\Sigma \V)$ is a torsion theory in the sense of~\cite[2.2]{iy}.
This is not the same as a torsion theory in the sense
of~\cite[Defn.\ 2.1]{br}, since there is no assumption of closure under
the suspension functor. It follows from Wakamatsu's Lemma
(see Lemma~\ref{wakamatsu}) and Lemma~\ref{doubleperp}
that, in our context, $(\add T,T^{\perp})$ is a torsion pair.

Nakaoka proves the following theorem:

\begin{theorem}
Let $\C$ be a triangulated category and $(\U,\V)$ a cotorsion pair in
$\C$. Let $\W=\U\cap \V$.
Let $\C^+$ be the full subcategory of $\C$ consisting of objects $C$
such that there is a distinguished triangle
$$W\to C \to \Sigma^{-1} V \to \Sigma W$$
in $\C$ with $W\in \W$ and $V\in \V$.
Let $\C^-$ be the full subcategory of $\C$ consisting of objects $C$
such that there is a distinguished triangle
$$\Sigma^{-1}U\to C \to W\to U$$
in $\C$ with $U\in \U$ and $W\in \W$.
Then $\C^+\cap \C^-$ contains $\U\cap \V$ and
$$\H \colon =\frac{\C^+\cap \C^-}{\U\cap \V}$$
is abelian.
\end{theorem}

If $(\U,\V)=(\add T,T^{\perp})$, where $T$ is a rigid object in $\C$
then it is easy to check that
$\C^+=\C$ and $\C^-=\Sigma^{-1}\C(T)$. Thus, in this case, Nakaoka's
result produces the subfactor abelian category
$\module \End_{\C}(T)^{\op}$ as in the Theorem~\ref{Tblanktheorem2}(a)
of Iyama-Yoshino, whereas in our approach we produce this category via
localisation.

Nakaoka already points out that the special case where
$\U=\V$ is a cluster-tilting object (or, more generally, a
cluster-tilting subcategory) recovers a result of Koenig-Zhu~\cite[3.3]{kz}
(see Theorem~\ref{Tblanktheorem2}(b)).

It is thus a natural question to ask whether the subfactor abelian category
$\H$ in Nakaoka's theorem can be obtained from the triangulated category via
localisation. However, the methods of this paper do not apply in this
situation since an appropriate generalisation of Lemma~\ref{equiv} does
not hold in general for a cotorsion pair.

In a sequel to this article~\cite{bm11}, we provide an alternative approach
to showing that $\C_{\widetilde{S}}$ is abelian,
by first considering the factor category $\C/{\X_T}$.
We prove that this category is preabelian, and that the family of regular 
maps $\R$ admits a calculus of left/right fractions, and moreover that we 
can recover $\module \Gamma$ by localising with respect to $\R$.
We note that although this approach is more closely related to the
work of Koenig-Zhu~\cite{kz} and Nakaoka~\cite{nakaoka}, it still does
not apply in the more general cotorsion theory set-up considered by Nakaoka.

\section{Examples}
In this section we give two examples to illustrate the main result. We
refer to~\cite{ars} for background on finite dimensional algebras and
their representation theory. The first example includes a map which
lies in $\widetilde{\S}$ but not in $\S$. In the second example, $\C$
does not contain any cluster-tilting objects; it is also interesting
to see how a module category of finite representation type arises from
the localisation of a cluster category with infinitely many indecomposable
objects. We also see that it is possible that the image of a indecomposable
module under localisation at $\widetilde{\S}$ can be decomposable.

\begin{example} \rm
Let $Q$ be the quiver:
$$\xymatrix{
1\ar[r] & 2\ar[r] & 3\ar[r] & 4.
}$$
The indecomposable modules over $kQ$ are determined by their
support on the vertices of $Q$. We denote by $M_{ij}$ the indecomposable
with support $\{i,i+1,\ldots ,j\}$. Let $T$ be the rigid module $T_1\amalg T_2\amalg T_3$,
where $T_1=M_{44}$, $T_2=M_{14}$ and $T_3=M_{11}$. Regarded as an object in the cluster
category $\C$ corresponding to $kQ$, $T$ is rigid, and it is easy to check that
$\End_{\C}(T)^{\op}$ is isomorphic to the algebra $\Lambda$ given by the following quiver with
a single relation:
$$\xymatrix{
1  & 2 \ar[l]^{}="a" & 3. \ar[l]^{}="b" \ar@{.}@/^9pt/ "a";"b"
}$$
The indecomposable $\Lambda$-modules are the simple modules $S_1,S_2,S_3$ and the
modules $\begin{matrix} S_2 \\ S_1 \end{matrix}$ and
$\begin{matrix} S_3 \\ S_2 \end{matrix}$
with dimension vectors $(1,1,0)$ and $(0,1,1)$ respectively.
If $M$ is any object in $\C$ then $$G(L_{\widetilde{\S}}(M))= H(M)=
\Hom_{\C}(T,M).$$
This enables us to compute the $\End_{\C}(T)^{\op}$-module $G(M)$
corresponding to each indecomposable object $M$ in $\C$.

The images of the indecomposable objects of $\C$ under $\Hom_{\C}(T,-)$
and the images of the irreducible maps between them are given in their
corresponding positions in the Auslander-Reiten quiver of
$\C$ below (drawn with the images of the
indecomposable projective modules on the left hand side, repeated on the
right hand side). 
$$
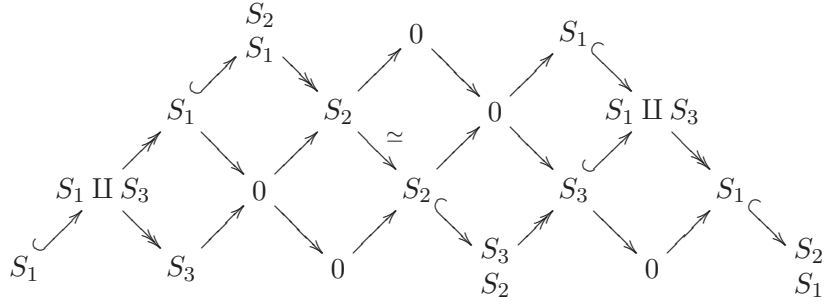
\xymatrix@!=0.2cm{
&&& {\begin{matrix} S_2 \\ S_1 \end{matrix}} \ar@{>>}[dr] && 0 \ar[dr] && S_1  \ar@{^{(}->}[dr] \\
&& S_1 \ar[dr] \ar@{^{(}->}[ur] && S_2 \ar^{\simeq}[dr] \ar[ur] && 0 \ar[ur] \ar[dr] && S_1 \amalg S_3 \ar@{>>}[dr] \\
& S_1 \amalg S_3 \ar@{>>}[dr] \ar@{>>}[ur] && 0 \ar[dr] \ar[ur] && S_2 \ar[ur] \ar@{^{(}->}[dr] && S_3 \ar[dr]  \ar@{^{(}->}[ur] && S_1 \ar@{^{(}->}[dr] \\
S_1 \ar@{^{(}->}[ur] && S_3 \ar[ur] && 0 \ar[ur] && {\begin{matrix} S_3 \\ S_2 \end{matrix}} \ar@{>>}[ur] && 0 \ar[ur] &&
{\begin{matrix} S_2 \\ S_1 \end{matrix}}
}$$

Thus we see that $\Hom_{\C}(T,-)$ is dense, as predicted by
Lemma~\ref{equiv}. We note that an indecomposable object in $\C$ can be sent
to a decomposable object by $\Hom_{\C}(T,-)$. We also remark that, in this example, all of
the irreducible maps in $\module \End_{\C}(T)^{\op}$ are of the form $\Hom_{\C}(T,u)$
where $u$ is irreducible in $\C$.

The right minimal almost split map $s \colon M_{44}\amalg \Sigma M_{24}\rightarrow M_{34}$ can be completed
to the triangle
$$\Sigma M_{34}\rightarrow M_{44}\amalg \Sigma M_{24}\rightarrow M_{34}\rightarrow M_{13}.$$
Since $\Sigma M_{34}$ lies in $\Sigma T^{\perp}$ and the map $M_{34}\rightarrow M_{13}$ factors
through $M_{23}\in \Sigma T^{\perp}$, it follows that $s\in S$ and thus is inverted by
the localisation functor $L_{\widetilde{\S}}$. Under the
functor $H=\Hom_{\C}(T,-)$,
the map
$s$ is mapped to an isomorphism $$\underline{s} \colon S_1\amalg S_3\rightarrow S_1\amalg S_3.$$

We note that the left minimal almost split map $t \colon M_{34}\rightarrow M_{33}\amalg M_{24}$ lies in
$\widetilde{\S}$ but not in $\S$. It is also inverted by $L_{\widetilde{\S}}$.
\end{example}

\begin{example} \rm
In this example we assume that $k$ is algebraically closed.
We consider the path algebra of the following quiver $Q$ of
type $\widetilde{A}_2$:
$$
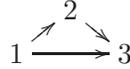
\xymatrix@C=0.3cm@R=0.1cm{
& 2 \ar[dr] & \\
1 \ar[ur] \ar[rr] && 3
}$$
The simple module $S_2$ is a rigid module at the mouth of a tube $\T$ of rank $2$ in the module
category of this quiver, giving rise to a rigid object $T$ in the cluster category $\C$.
Then $\Gamma=\End_{\C}(T)^{\op}$ is given by a quiver with a single vertex and a loop whose square
is zero. The indecomposable $\Gamma$-modules are the uniserial projective module, $P$,
and the simple module $S$. Note that $\Hom_{\C}(T,-)$ kills all of the objects in the
homogeneous tubes. The images of the other indecomposable objects in $\C$ and the
images of the irreducible maps between them under $\Hom_{\C}(T,-)$ are given in their
corresponding positions in the Auslander-Reiten quiver below.

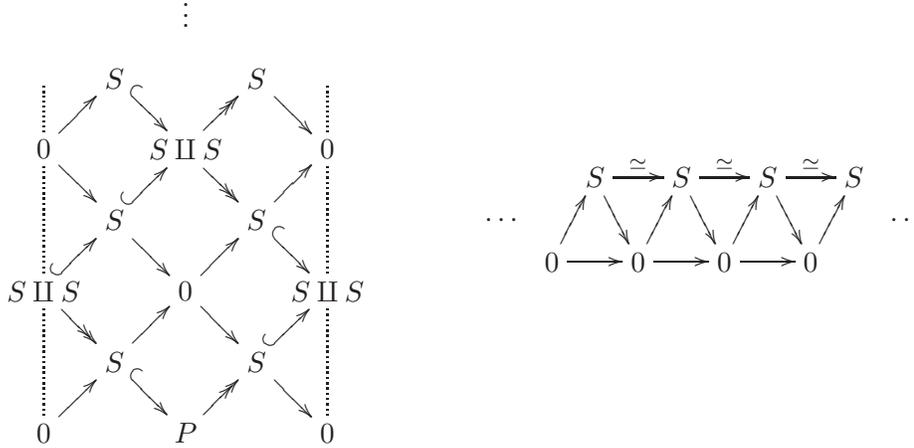
\begin{tabular}{m{6cm}m{6cm}}
$\xymatrix@!=0.1cm{
&&\vdots && \\
&S \ar@{^{(}->}[dr] && S \ar[dr] & \\
0 \ar@{.}[u] \ar[ur] \ar[dr] && S\amalg S \ar@{>>}[dr] \ar@{>>}[ur] && 0 \ar@{.}[u] \\
&S \ar@{^{(}->}[ur] \ar[dr] && S \ar@{^{(}->}[dr] \ar[ur] \\
S\amalg S \ar@{.}[uu] \ar@{>>}[dr] \ar@{^{(}->}[ur] && 0 \ar[ur] \ar[dr] && S\amalg S \ar@{.}[uu] \\
&S \ar@{^{(}->}[dr] \ar[ur] && S \ar@{^{(}->}[ur] \ar[dr] \\
0 \ar@{.}[uu] \ar[ur] && P \ar@{>>}[ur] && 0 \ar@{.}[uu]
}$
&
$\xymatrix@C=0.12cm@R=0.12cm{
&& S \ar^{\simeq}[rr] \ar[ddr] && S \ar^{\simeq}[rr] \ar[ddr] &&
S \ar^{\simeq}[rr] \ar[ddr] && S & \\
\cdots  &&&&&&&&& \cdots \\
& 0 \ar[rr] \ar[uur] && 0 \ar[rr] \ar[uur] && 0 \ar[rr] \ar[uur]
&& 0 \ar[uur]
}
$
\end{tabular}

\vskip 0.1cm
We note that the cluster category $\C_{\T}$ of a tube $\T$ (defined
using~\cite[Section 3]{bmrrt})
is a full subcategory of the cluster category of $\module kQ$ and has been
studied in~\cite{bkl1,bkl2,bmv,vatne}.
It does not contain any cluster-tilting objects.
Theorem~\ref{main} applies to the case of a maximal rigid
object $T$ in $\C_{\T}$. It follows that the module categories over
the algebras studied in~\cite{vatne}, i.e.\ the endomorphism algebras of
maximal rigid objects in $\C_{\T}$, can be obtained as localisations of $\C_{\T}$.
The above example gives rise to the case of a rank
two tube, with $T$ a maximal rigid object in $\C_{\T}$. We see that,
as in~\cite[Thm. 4.9]{vatne}, $\Hom_{\C_{\T}}(T,M)$ can be decomposable even if
$M$ is indecomposable. It follows from Theorem~\ref{main} that the same is
true for $L_{\widetilde{\S}}(M)$.
\end{example}

\end{document}